\documentclass[12pt]{article}

\usepackage{amsfonts}
\usepackage{amsmath, amsthm}
\usepackage{amssymb}
\usepackage{indentfirst}
\usepackage{graphicx}
\graphicspath{}
\DeclareGraphicsExtensions{.pdf,.png,.jpg}
\usepackage[english]{babel}
\usepackage{latexsym}
\renewcommand{\le}{\leqslant}
\renewcommand{\ge}{\geqslant}

\newcommand{\F}{\mathcal F}
\newcommand{\FF}{\mathbb F}
\newcommand{\A}{\mathcal A}
\newcommand{\fG}{\mathfrak{G}}

\newtheorem{theorem}{Theorem}[section]

\newtheorem{lemma}[theorem]{Lemma}

\newtheorem{claim}[theorem]{Claim}

\begin{document}
\title{Chromatic numbers of Kneser-type graphs}
\author{Dmitriy Zakharov}
\date{}
\maketitle
\begin{abstract}
Let $G(n, r, s)$ be a graph whose vertices are all $r$-element subsets of an $n$-element set, in which two vertices are adjacent if they intersect in exactly $s$ elements. In this paper we study chromatic numbers of $G(n, r, s)$ with $r, s$ being fixed constants and $n$ tending to infinity. Using a recent result of Keevash on existence of designs we deduce an inequality $\chi(G(n, r, s)) \le (1+o(1))n^{r-s} \frac{(r-s-1)!}{(2r-2s-1)!}$ for $r > s$ with $r, s$ fixed constants. This inequality gives sharp upper bounds for $r \le 2s+1$.  Also we develop an elementary approach to this problem and prove that $\chi(G(n, 4, 2)) \sim \frac{n^2}{6}$ without use of Keevash's results. 

Some bounds on the list chromatic number of $G(n, r, s)$ are also obtained.
\end{abstract}

\section{Introduction}

Let $G(n, r, s)$ be a graph whose vertices are all $r$-element subsets of an $n$-element set, in which two vertices are adjacent if they intersect in exactly $s$ elements. If we let $s = 0$ then we obtain a classical family of graphs, namely Kneser graphs $KG(n, r)$. On the other hand, if $s=r-1$ then we obtain the sequence of Johnson graphs $J(n, r)$. So, $G(n, r, s)$ may be regarded as a generalization of both Kneser and Johnson graphs. These graphs have many applications in combinatorial geometry (\cite{FW}, \cite{KK}, \cite{CKR}), coding theory (\cite{WS}, \cite{R}) and Ramsey theory (\cite{Nagy}). In many applications it is important to estimate the chromatic number of $G(n, r, s)$. Recall that the chromatic number $\chi(G)$ of a graph $G$ is the minimal number $k$ such that the vertices of $G$ can be colored in $k$ colors in such a way that there is no monochromatic edge. In this paper we study upper bounds on $\chi(G(n, r, s))$ where $r > s$ are fixed constants and $n$ tends to infinity (see \cite{BKK} for known results, also see \cite{P}, \cite{Z}). Some recent related results can be found in \cite{KBC}, \cite{KKup}, \cite{Kup}, \cite{C}, \cite{SR}, \cite{ZR}.

Denote by $[n] = \{1, \ldots, n\}$ the set of first $n$ positive integers, by ${X \choose r}$ the set of all $r$-element subsets of $X$. It is convenient to assume that the set of vertices of $G(n, r, s)$ coincides with ${[n] \choose r}$. 

The chromatic number $\chi(G(n, r, s))$ behaves quite differently depending on the relation between $r$ and $s$:
if $r > 2s+1$ then $\chi(G(n, r, s)) = \Theta(n^{s+1})$ and if $r \le 2s+1$ then $\chi(G(n, r, s)) = \Theta(n^{r-s})$. The reason is that there are two different types of independent sets\footnote{Recall that the set of vertices of a graph is independent if any two of its vertices are not adjacent. The independence number $\alpha(G)$ is the size of the largest independent set of $G$. For any graph $G$ one has $\chi(G) \ge |V|/\alpha(G)$.} in these graphs. 
The first example of an independent set in $G(n, r, s)$ is a {\it star}: the family of sets containing a fixed $(s+1)$-set. It has cardinality ${n-s-1 \choose r-s-1}$ which is asymptotically  $\frac{n^{r-s-1}}{(r-s-1)!} = \Theta(n^{r-s-1})$ if $r, s$ are fixed. It was proved by Frankl and F{\" u}redi \cite{FF} that if $r > 2s+1$ and $n$ is sufficiently large then the star is the maximal independent set in $G(n, r, s)$. Recently the result has been significantly extended in \cite{KL} to the regime where $s$ is fixed and $C < r < n / C$ for some absolute constant $C$.

In the case $r \le 2s+1$ the maximal independent sets can not be classified is any reasonable way. All known constructions come from design theory (the connection will be indicated in Section \ref{s1}) and explicit constructions are known in very special cases only. In particular, R{\" o}dl \cite{Ro}, using a probabilistic method, showed that $\alpha(G(n, r, s)) \ge (1+o(1))n^s \frac{(2r-2s-1)!}{r!(r-s-1)!} = \Theta(n^s)$ (see \cite{BKK} for more details). The matching upper bound is known if $(r-s)$ is a power of a prime \cite{FW} and it is a major open problem to obtain the same upper bound without the  assumption that $(r-s)$ is a power of a prime.

The intermediate case $r = 2s+1$ is of particular interest. For this choice of parameters both star and design-type constructions give lower bounds of the same order of magnitude $\Theta(n^s)$, so one may wonder which construction gives a better estimate. Miraculously, both constructions result in exactly the same bound $\alpha(G(n, 2s+1, s)) \ge (1+o(1))\frac{n^s}{s!}$ (although, the $o$-term is slightly better in the design-type construction).

It is not difficult to see that the chromatic number of $G(n, r, s)$ has the order of magnitude $\Theta(\frac{n^r}{ \alpha(G(n, r, s))})$ which explains the aforementioned dichotomy. 
The problem now is to determine the constant $C$ (depending on $r, s$) such that $\chi(G(n, r, s)) \sim Cn^{\min\{s+1, r-s\}}$. 

%Now we turn back to the chromatic numbers of $G(n, r, s)$.
For $r > 2s+1$ all known upper bounds on $\chi(G(n, r, s))$ are obtained using Tur{\'a}n numbers (see \cite{BKK}, \cite{S}). The author is unaware of any improvements of the trivial inequality $\chi(G(n, r, s)) \ge {n \choose r}/\alpha(G(n, r, s)) \sim n^{s+1} \frac{(r-s-1)!}{r!}$ for $r > 2s+1$, except for the case $s =0$ in which the chromatic number is known exactly and equals to $\chi(G(n, r, 0)) = n - 2r+ 2,~(n\ge 2r)$ (this is a celebrated result of Lov{\' a}sz \cite{L}).
% although, it does not seem to be sharp (for instance, the celebrated result of Lov{\' a}sz \cite{L} is that $\chi(G(n, r, 0)) = n - 2r + 2$ for all $n \ge 2r$, which is much larger than trivial  bound $\frac{n}{r}$).  

In this paper we will consider the region $r \le 2s+1$. The simplest upper bound on the chromatic number of $G(n, r, s)$ is this case is the maximal degree bound: $\chi(G(n, r, s)) \le \Delta(G(n, r, s)) + 1 \sim n^{r-s} \frac{r!}{s! ((r-s)!)^2}$ which already gives the correct order of magnitude. Some non-trivial estimates were obtained by the author in \cite{Z}, for instance, $\chi(G(n, r, s)) \le (1+o(1)) n^{r-s}$ which improves the maximal degree bound if $(r-s)$ is less than $\sqrt{s}$. The first result of the present paper is the following sharp result.

\begin{theorem}\label{cor0}
Let $r > s$. Then $\chi(G(n, r, s)) \le (1+o(1))n^{r-s} \frac{(r-s-1)!}{(2r-2s-1)!}$ as $n \rightarrow \infty$.
\end{theorem}

Note that if $r \le 2s+1$ and $(r-s)$ is a power of a prime then the bound in Theorem \ref{cor0} coincides with known lower bounds. In fact, Theorem \ref{cor0} is a simple corollary of recent results of Keevash \cite{K2}. Results of \cite{K2} require $n$ to be extremely large compared to $r, s$ whereas in most of the  applications of graphs $G(n, r, s)$ one needs to consider $r, s$ growing with $n$. Moreover, in applications to combinatorial geometry one typically requires $r, s$ to grow linearly with $n$.  %\footnote{However, in this regime the chromatic number of $G(n, r, s)$ is more or less defined by $\alpha(G(n, r, s))$: for any vertex-transitive graph $G$ we have a chain of inequalities $\frac{|V(G)|}{\alpha(G)} \le \chi(G) \le \frac{|V(G)| \log(|V(G)|)}{\alpha(G)}$. If $r, s$ grow linearly then the factor of $\log{n \choose r}$ usually becomes negligible in estimates involving  $\chi(G(n, r, s))$.}

The main aim of this paper is to present a different approach to the problem. Namely, we develop a new elementary approach and solve the special case $(r, s) = (4, 2)$, which is the first unsolved case in the region $r \le 2s+1$. The best known upper bound on the chromatic number of $G(n, 4, 2)$ is $\frac{n^2}{2}+100n$ \cite{BKK}. %The reason of this particular choice is that this is the smallest pair $(r, s)$ with $r \le 2s+1$ for which the chromatic number was unknown.
Note that if we consider the family of vertices of $G(n, 4, 2)$ which contain element $\{1\}$, the induced subgraph will be isomorphic to $G(n-1, 3, 1)$. This means that any proper coloring of $G(n, 4, 2)$ will automatically lead to a proper coloring of $G(n, 3, 1)$, so it is important to understand how to color $G(n-1, 3, 1)$ first. In Section \ref{s2} we provide a simple proof of the inequality $\chi(G(n, 3, 1)) \le \frac{(n-1)(n-2)}{6}$ for $n = 2^t$ being a power of $2$. 
	
%Let us now focus on the special case $r=4, s=2$: it is the smallest pair $(r, s)$ with $r \le 2s+1$ for which asymptotic of $\chi(G(n, r, s))$ was unknown. By Theorem \ref{cor0} and known lower bound we see that $\chi(G(n, 4, 2)) \sim \frac{n^2}{6}$. We give an elementary proof in this fact without use of Keevash's results:

\begin{theorem}\label{main}
$\chi(G(n, 4, 2)) \le (1+o(1)) \frac{n^2}{6}.$
\end{theorem}

Of course, Theorem \ref{main} is a particular case of Theorem \ref{cor0} but techniques developed in the proof of this result may be of independent interest. %Although, our construction is not completely explicit and uses probabilistic method at some places, it provides a much more structured coloring of $G(n, 4, 2)$. We discuss this issue in more detail in Section \ref{remarks}. Also we discuss the possibility to extend the result to a wider range of parameters. 

In addition, in order to prove Theorem \ref{main} we need to estimate the {\it list chromatic number} of $G(n, r, s)$. Recall that the list chromatic number $\chi_{list}(G)$ of a graph $G$ is the minimal number $k$ such that for any arrangement of sets $L(v), ~v \in V(G)$, each $L(v)$ of size $k$, there are colors $c(v) \in L(v)$ such that each edge is not monochromatic. In the end of Section \ref{ss6} we prove the following.

\begin{lemma} \label{listi} 
Fix $r, s$ and  let $n \rightarrow \infty$, then 
$\chi_{list}(G(n, r, s)) = O(n^{s+1} \log n) $.
\end{lemma}

It would be interesting to obtain more estimates on $\chi_{list}(G(n, r, s))$ in various asymptotic regimes as well.

In Section \ref{s1} we prove Theorem \ref{cor0} and in Section \ref{ss2} we prove Theorem \ref{main}, Section \ref{remarks} contains some final remarks. %In Section \ref{remarks} we compare our method with Keevash's and discuss possible extensions of Theorem \ref{main}. 

\vskip 1cm

{\bf Acknowledgements.} I would like to acknowledge prof. Raigorodskii for introducing me to the subject and for his constant encouragement. I would like to thank referees for helpful suggestions and comments.

\section{Proof of Theorem \ref{cor0}}\label{s1}

Recall from \cite{K2} that a family $\F$ of $r$-element subsets of an $n$-set $X$ forms an $(n, r, s)$-{\it design} if every $s$-element subset of $[n]$ belongs to exactly one element of $\F$. A {\it complete resolution} of ${[n] \choose r}$ is a partition of ${[n] \choose r}$ into $(n, r, r-1)$-designs, each of which is partitioned into $(n, r, r - 2)$-designs, and so on, down to $(n,r, 1)$-designs.  

Keevash \cite{K2} proved that a complete resolution of ${[n] \choose r}$ always exists provided that $n \equiv r \pmod {{\rm gcd}[r]}$ and $n$ is sufficiently large.  

In particular, given $n$ as above for any $p>q$ there is an $(n, r, p)$-design which can be decomposed into $(n, r, q)$-designs.

We relate decompositions of designs and colorings of $G(n, r, s)$ by the following simple claim.
\begin{claim}
Suppose that there is an $(n, 2r-s-1, r)$-design which can be decomposed into $(n, 2r-s-1, s)$-designs. Then $\chi(G(n, r, s)) \le N := \frac{{n \choose r}}{{2r-s-1 \choose r}}/\frac{{n \choose s}}{{2r-s-1 \choose s}}$.
\end{claim}
\begin{proof}
If possible, take a decomposition of an $(n, 2r-s-1, r)$-design into $(n, 2r-s-1, s)$-designs: $\mathcal D = \mathcal D_1 \cup \ldots \cup \mathcal D_N$ where $N$ is as in the statement of the claim. As $\mathcal D$ is an $(n, 2r-s-1, r)$-design, this decomposition induces a decomposition of ${[n] \choose r}$ into $N$ classes: an $r$-set $A$ belongs to a class $i$ if there is a set $X \in \mathcal D_i$ such that $A \subset X$. If two $r$-sets $A, B$ belong to the same class $i$ then there are $X, Y \in \mathcal D_i$ such that $A \subset X, B \subset Y$. Because $\mathcal D_i$ is an $(n, 2r-s-1, s)$-design we have either $X=Y$ and $|A \cap B| \ge s+1$ or $|X \cap Y| \le s-1$ and $|A \cap B| \le s-1$. In both cases $A$ and $B$ are not connected by an edge and we have constructed a proper coloring of $G(n, r, s)$ using $N$ colors.

\end{proof}

For any sufficiently large $n$ we take $n'$ such that $n+r! \ge n' \ge n$ and $n' \equiv r \pmod{gcd[r]}$. By Keevash's theorem and the above claim we obtain
$$
\chi(G(n, r, s)) \le \chi(G(n', r, s)) \le (1+o(1))n'^{r-s} \frac{r! (r-s-1)!}{r!(2r-s-1)!} / \frac{s! (2r-2s-1)!}{s! (2r-s-1)!}=(1+o(1))n^{r-s} \frac{(r-s-1)!}{(2r-2s-1)!}.
$$

\section{Proof of Theorem \ref{main}}\label{s3}

\subsection{Sketch of the proof} \label{s2}

In the proof of Theorem \ref{cor0} we colored graph $G(n, r, s)$ using a decomposition of a certain design into other designs. In the case of $G(n, 4, 2)$ our strategy will be the same but the main difficulty is to construct the required designs without use of heavy machinery. To color $G(n, 4, 2)$ we need to build a decomposition of an $(n, 5, 4)$-design into $(n, 5, 2)$-designs, but we will be able to provide only an approximate version of this decomposition and this will suffice for our purposes. We call a family $\F \subset {[n] \choose r}$ an approximate $(n, r, s)$-design if each $s$-element set is contained in at most one set from $\F$ and if the number of $s$-element subsets which are not contained in any set of $\F$ is $o({n \choose s})$.

Our proof is largely inspired by the proof of the following simple result about $G(n, 3, 1)$:

\begin{theorem}[\cite{BKR}]\label{BKR}
Let $n = 2^t$. Then $\chi(G(n, 3, 1)) \le \frac{(n-1)(n-2)}{6}$.
\end{theorem}

Note that this bound is tight (see \cite{BKK}).

\begin{proof}
Identify  $[n]$ with $V = \FF_2^t$. Let us say that two triples of vectors are equivalent if one of them can be obtained from another by a translation. It is easy to see that there are exactly ${|V| \choose 3} / |V| = \frac{(n-1)(n-2)}{6}$ equivalence classes. 

Now we prove that each class form an independent set, i.e. any two sets from the same class intersect in an even number of elements. Take a pair of equivalent triples $\{a, b, c\}, \{a, d, e\}$, by definition we have $\{a, b, c\}+v = \{a, d, e\}$ for some $v \in V\setminus 0$. Note that $a+b \neq a+d, a+e$, so $a+b = d+e$ because $a+b = (a+v)+(b+v) \in \{a+d, a+e, d+e\}$. By the same reasoning $a+c = d+e$ and, therefore, $b=c$. A contradiction.\footnote{The presented proof differs from that given in \cite{BKR}. I found this proof on the cite of Moscow Mathematical Olympiad (https://olympiads.mccme.ru/mmo/2012/75mmo.pdf, page 44, in Russian). The original proof is by induction on the power of $2$, and it generalizes to the bound $\chi(G(n, 3, 1)) \le \frac{n(n-1)}{6} +cn$ for arbitrary $n$ \cite{BKR}.  Consequently, in \cite{BKK} the same idea was applied to $G(n, 4, 2)$ and the bound $\chi(G(n, 4, 2)) \le \frac{n^2}{2}+100n$ was obtained. The last inequality is the best previously known upper bound for the chromatic number of $G(n, 4, 2)$.}

\end{proof}

\vskip 0.5cm

Now we describe ideas of the proof. %As in the proof of Theorem \ref{cor0} we want to reduce the problem of coloring $G(n, r, s)$ to the problem of coloring an $(n, 2r-s-1, r)$-design. In our case this will be an $(n, 5, 4)$-design. 

%Basically, in the above proof we considered the action of $\FF_2^t$ on the set of all triples of $\FF^t_2$ and checked that each orbit forms an independent set. 
%In the proof of Theorem \ref{main} we will consider the set $\FF_p^2 \setminus 0$ instead of $\FF_2^t$ and the action of general linear group on subsets of $\FF_p^2 \setminus 0$ instead of the action of $\FF_2^t$ on its subsets. 

At first we note that we may assume that $n=p^2-1$, where $p$ is a prime. So we can identify $[n]$ with the set $\FF_p^2 \setminus 0$. Denote $V = \FF_p^2,~ \fG = GL_2(\FF_p)$ and consider the family $\A = \{\{v_1, \ldots, v_5\}:~v_i \in V\setminus 0,~v_1 + \ldots + v_5 = 0,~ v_1, \ldots, v_5 \text{ are pairwise non-collinear}\}$. It is easy to see that $\A$ is an approximate $(n, 5, 4)$-design and $|\A| \sim \frac{n^4}{120}$.

%We need to take care of $4$-element subsets of $V\setminus 0$ which are not contained in any set from $\A$, i.e. we need to ensure that these vertices can be colored using only $o(n^2)$ colors. This is done in Section \ref{ss6}. 

First, we show how to color in a small number of colors all $4$-element sets which are contained in  sets from $\A$. In order to color them properly it is enough to color sets from $\A$ in such a way that two sets of the same color intersect in at most $1$ element. To do this, we divide $\A$ into orbits under the action of $\fG$: $\A=\A_1 \cup \ldots \cup \A_l$. We note that $l \sim \frac{n^2}{120}$ because $|\A|\sim \frac{n^4}{120}$ and $|\fG| \sim n^2$. The idea is to color each orbit separately from the others. The main observation is that inside of each orbit the induced subgraph has a very special structure: to see this, let $A_1, A_2 \in \A_i$ be an adjacent pair of vertices from the same orbit, that is $|A_1 \cap A_2| \ge 2$; since $\A_i$ is an orbit, there is $g\in \fG$ such that $A_2 = gA_1$, so the map $g$ maps a pair of elements of $A_1$ to another pair of elements of $A_1$. Since any linear map on a vector space is completely determined by its action on a basis elements, this in particular implies that the degree of any vertex in $\A_i$ is always at most $20 \cdot 19$.  
%but the number of $g \in \fG$ such that $|A_1 \cap gA_1| \ge 2$ is at most $20\cdot 19$ because a linear map is determined by the image of any two linearly independent vectors (one can choose an ordered pair $(\{x_1, x_2\}, \{y_1, y_2\})$ of couples of elements of $A_1$  in exactly $20 \cdot 19$ ways, each linear map corresponds to one of these pairs). 
Let $E_i$ be the set of $g \in \fG$ such that $|gA_1 \cap A_1| \ge 2$. We conclude that the graph induced on $\A_i$ has the structure of a Cayley graph on $\fG$ with the set of generators $E_i$.
%Denote the set of such $g$ by $E_i$, we see that the induced subgraph in each orbit becomes isomorphic to a Cayley graph on $\fG$ with generators $E_i$. 

For most of the orbits $\A_i$ we are able to gain enough control on the local structure of this Cayley graph using algebraic tools. This allows us to construct a proper coloring of a large neighborhood of any vertex in the orbit $\A_i$. But there may be some ``global" obstructions to extend these ``local" colorings to the whole orbit. %This allows us to construct a ``local" coloring of $\A_i$ in the required number of colors using a coloring which looks like a ``local homomorphism" to the group $\mathbb{Z}_{20}$. Now the goal is to glue a coloring of the whole orbit from these ``local" colorings. 
To overcome this, we construct a reasonably small set $\A^{wall} \subset \A$ such that for almost all orbits $\A_i$ the set $\A_i \setminus \A^{wall}$ splits into tiny components for which the local coloring can be applied. It remains to color the set $\A^{wall}$ and all remaining orbits which have not yet been colored. This can be done by estimating the maximal degrees of the corresponding induced subgraphs and using other crude bounds. In particular, the following simple observation is used.

\begin{claim}%\label{general}
Let $H$ be a graph. Suppose that the vertex set of $H$ is covered by a system of subsets: $V(H) = \bigcup_{i=1}^{m} A_i$. Suppose that for any $v \in V(H)$ the number of edges between $v$ and the set $\bigcup_{A_i \not \ni v} A_i$ is at most $d$. Suppose that for any $i$ $\chi_{list}(H|_{A_i}) \le l$ holds. Then $\chi(G) \le l+d$.  
\end{claim}

%Fix an orbit $\A_i$ and its representative $A \in \A_i$. Let $S(A)$ be a set of all $g\in \fG$ for which $|A\cap gA| \ge 2$, so we have $|gA \cap hA| \ge 2$ if and only if $h^{-1}g \in S(A)$. Therefore, each orbit has a structure of a Cayley graph on $\fG$ with edges arising from $S(A)$. It can be shown that for the most orbits $\A_i$ the elements of $S(A)$ locally behave like elements of a free group. The union of orbits which are not of this kind we denote by $\A^{short}$ and color them separately (it is done in Lemma \ref{lm2} using Lemma \ref{short}). 

%Then we construct a set $\A^{wall} \subset \A$ which cuts all generic orbits $\A_i$ into small components in which we can control behavior of induced subgraphs. The Wall is constructed in Lemma \ref{wall} and colored in $o(n^2)$ colors in Lemma \ref{lm2}. Now large part of $\A_i$ for almost all $i$ is divided into small components. In each component $\mathcal C$ we provide an explicit coloring $c: \mathcal C \rightarrow \mathbb Z_{20}$ using a certain map $\psi: S(A) \rightarrow \mathbb Z_{20}$ and extending it on the whole component $\mathcal C$ by linearity (this part is given in the Lemma \ref{lm1}). To prove that coloring is correct we use Lemma \ref{trivial} which basically states that the color $c(v)$ of a vertex $v$ is independent from the choice of path going to it.  

\subsection{Beginning of the proof}\label{ss2}

By Prime Number Theorem, for any natural $n$ there is a prime $p$ such that $p^2-1\ge n$ and $p^2-1 \sim n$. Because $G(n, 4, 2)$ can be embedded into $G(p^2-1, 4, 2)$ we may assume that $n=p^2-1$. Consequently, we may identify $[n]$ with $\FF_p^2 \setminus 0$. Denote $V=\FF_p^2$ and let $\fG = GL_2(\FF_p)$ be the general linear group of $V$.

First, we introduce some notation. Let $\A = \{ \{v_1, \ldots, v_5\}:~v_i \in V\setminus 0,~v_1 + \ldots + v_5 = 0,~v_i,\, v_j\, \text{are not collinear}\}$ and, similarly, $\A_{ord} = \{ (v_1, \ldots, v_5):~v_i \in V\setminus 0,~v_1 + \ldots + v_5 = 0,~v_i,\, v_j\, \text{are not collinear}\}$ (note that $\A \subset {V \choose 5}$ but $\A_{ord} \subset V^5$). There is a natural projection $\pi: \A_{ord} \rightarrow \A$. 

Throughout the proof we use indexations like $d_{ij}$ or $\prod_{ij}$ omitting the range of $i, j$. Unless otherwise specified, it means that the range is $1 \le i, j \le 5$ and $i \neq j$. For any family $\F$ and a set $S$ we denote by $\F(S)$ the subfamily of sets from $\F$ containing $S$.

For linearly independent vectors $a, b \in V$, denote by $g_{a, b}$ the linear map which maps the standard basis $e_1, e_2$ to $a, b$. Note that $g_{a, b} \in \fG$ and the matrix of this operator is just $(a, b)$. Denote by $g_{ij}: \A_{ord} \rightarrow \fG$ a function which maps a sequence $A=(a_1, \ldots, a_5) \in \A_{ord}$ to the operator $g_{a_i, a_j}$, i.e. $g_{ij}(A) = g_{a_i, a_j}$.

A {\it dependence $\omega$ of length $t$ } is a sequence $(d_{ij})$ (here $1 \le i, j \le 5$, $i \neq j$) of integers such that $\sum_{i \neq j} d_{ij}=0$ and $\sum_{i < j} |d_{ij}+d_{ji}| = 2t$. We think of each dependence $\omega$ as of a map $\A_{ord} \rightarrow \FF_p$ defined as follows:

\begin{equation} \label{omega}
 \omega(A) = \prod_{i \neq j} \det(g_{ij}(A))^{d_{ij}},
\end{equation}

A dependence $\omega$ is called trivial if $\omega(A) = 1$ for any $A\in \A_{ord}$. Otherwise $\omega$ is called nontrivial. Given two dependencies $\omega = (d_{ij})$ and $\omega' = (d'_{ij})$ one can define the product: $\omega \omega' := (d_{ij}+d'_{ij})$. In view of (\ref{omega}), it will be sometimes convenient to denote a dependence $\omega = (d_{ij})$ as $\omega = \prod \det( g_{ij})^{d_{ij}}$. 

Fix $t = n^{0.01}$, denote by $\A^{short}_{ord}$ the set of all sequences $A \in \A_{ord}$ such that there is a {\it nontrivial} dependence $\omega$ of length at most $t$ and $\omega(A)=1$. Denote $\A^{short} = \pi(\A_{ord}^{short})$ and let $\A^{long} := \A \setminus \A^{short}$ and $\A^{long}_{ord} := \A_{ord} \setminus \A_{ord}^{short}$.

Now we define sets $\A^{wall}_{ord}$ and $\A^{wall}$. Let $E$ be the set of functions $ \A_{ord} \rightarrow \fG$ of the form $g_{ij}g_{kl}^{-1}(A) := g_{ij}(A)g_{kl}^{-1}(A)$ where $i \neq j, k \neq l$ and $(i, j) \neq (k, l)$. The motivation of this definition is that %the neighborhood of a vertex $hA$ in the orbit $\fG A$ is exactly the set $\{ e(A)hA | e \in E\}$. Indeed, let 
two element $hA, h'A \in \fG A$ of the orbit of $A$ intersect in at least two elements if and only if $h^{-1}h' = e(A)$ for some $e \in E$.

Each element $g_{ij}g_{kl}^{-1}$ of $E$ determines a dependence $\omega = \det g_{ij} g_{kl}^{-1}$ that is $\omega$ is a sequence $(d_{rs})$ with $d_{ij} = 1, d_{kl} = -1$ and all other coordinates are set to be zero. For a sequence of elements $g_1, \ldots, g_m \in E$ we denote by $g_1 g_2 \ldots g_m: \A_{ord} \rightarrow \fG$ a function which maps a sequence $A$ to the operator $g_1(A)g_2(A) \ldots g_m(A)$, also we denote by $\det(g_1g_2 \ldots g_m)$ the product of dependencies $\det(g_i)$.

Abusing the notation, for an arbitrary family $\F$ and elements $\alpha, \beta$ denote $\F(\alpha, \beta) := \F(\{\alpha, \beta\})$.

In Section \ref{ss4} we will prove the following lemma.  
\begin{lemma}\label{lmin}
There are sets $\A^{wall}_{ord} \subset \A_{ord}$ and $\A^{wall}  = \pi(\A^{wall}_{ord})$ such that: 

1. Choose any $A \in \A_{ord}^{long}$ and $g_1, \ldots, g_m \in E$. Suppose that the dependence $\omega = \det(g_1 g_2 \ldots g_m)$ is of length at least $t/3$ and $\omega(A)=1$. Then for some $i$ we have $g_1 g_{2} \ldots g_i(A)A \in \A^{wall}_{ord}$.

2. For any pair of vectors $\alpha, \beta$ we have $|\A^{wall}(\alpha, \beta)| = o(n^2)$.
\end{lemma} 

Let us briefly explain the meaning of these definitions and the role of dependencies in the proof. For the sake of simplicity, here we ignore the difference between ordered and unordered families. Let $\A_1 \subset \A^{long}$ be an orbit of the action of $\fG$ and let $A$ be a representative of $\A_1$. Let us connect two different sets from $\A_1$ if they intersect in at least two elements. By definition of $E$, two sets $h_1A$ and $h_2A$ are connected by an edge if there is $g \in E$ such that $h_2^{-1}h_1 = g(A)$. Thus, for each cycle $\mathcal C = \{h_1A, \ldots, h_mA\} \subset \A_1$ of the resulting graph we can construct a dependence $\omega = \det (g_1 \ldots g_m)$ where $g_i \in E$ are such that $g_i(A) = h_{i-1}^{-1}h_i$.
Note that $\omega(A) = 1$ and vertices of $\mathcal C$ can be represented as follows:
$$
\mathcal C = \{ h_1 A, h_1 g_1(A) A, h_1 g_1 g_2(A) A, \ldots, h_1 g_1 \ldots g_{m-1}(A) A\}
$$
We call a cycle $\mathcal C$ nontrivial if the corresponding dependence is nontrivial. Now the first conclusion of Lemma \ref{lmin} implies that any nontrivial cycle $\mathcal C$ in a ``long" orbit $\A_1$ must intersect $\A^{wall}$ (to see this, note the identity $h_1 g_1(A)A = g_1(h_1 A) h_1 A$). Thus, the graph induced on the set $\A_1 \setminus \A^{wall}$ contains only ``trivial" cycles which will allow us to color this set optimally.

Finally, we form a set $\A_{good} = \A \setminus (\A^{short} \cup \A^{wall})$ and construct a graph $G$ on $\A_{good}$ in which two sets are adjacent if they intersect in at least two places.
 
\begin{lemma} \label{lm1}
$\chi(G) \le (1+o(1)) \frac{n^2}{6}$.
\end{lemma}

Given this lemma we can color in $(1+o(1)) \frac{n^2}{6}$ colors vertices of $G(n, 4, 2)$ which are contained in some set from $\A_{good}$ (just like in the proof of Theorem \ref{cor0}). Denote the set of all remaining vertices by $U$, it remains to prove that this set can be colored in a small number of colors.

\begin{lemma} \label{lm2}
$\chi(G(n, 4, 2)|_{U}) = o(n^2)$.
\end{lemma}

Clearly, the combination of these lemmas implies Theorem \ref{main}.

The rest of the proof is organized as follows. In Section \ref{ss3} we prove auxiliary results about trivial and nontrivial dependencies. In Section \ref{ss4} we construct the set $\A^{wall}$. In Sections \ref{ss5} and \ref{ss6} we prove Lemmas \ref{lm1} and \ref{lm2} respectively.

\subsection{Dependencies}\label{ss3}

We begin with a simple observation.
\begin{claim}\label{dep} 
Take $A \in \A_{ord},~h\in \fG$ and arbitrary dependence $\omega$. Then $\omega(hA) = \omega(A)$. 
\end{claim}
\begin{proof}
Note that for any $g \in \fG$~ $g_{ha, hb} = hg_{a, b}$ and that $g_{ij}(hA) = g_{ha_i, ha_j}=hg_{a_i, a_j} = hg_{ij}(A)$.  So, 
$$
\omega(hA) = \prod_{i \neq j} \det(g_{ij}(hA))^{d_{ij}} = \prod_{i \neq j} \det(hg_{ij}(A))^{d_{ij}} =\det(h)^{\sum d_{ij}} \prod_{i \neq j} \det(g_{ij}(A))^{d_{ij}} = \omega(A).
$$
\end{proof}

Denote by $\A_{ord}(\alpha, \beta)$ the set of sequences $A \in \A_{ord}$ such that $A = (\alpha, \beta, x_1, x_2, x_3)$ for some $x_i \in V$. 
The next lemma states that the set of sufficiently degenerate sequences $A \in \A_{ord}$ is rather sparse.

\begin{lemma}\label{short}
$|\A^{short}_{ord}(\alpha, \beta)| = O(n^{1.7})$ for any linearly independent $\alpha, \beta \in V \setminus 0$.
\end{lemma}

\begin{proof}
Note that there are at most $2^{30}t^{10}$ dependencies of length at most $t$ which determine different functions $\A \rightarrow \FF_p$. Indeed, take a dependence $\omega = (d_{ij})$ of length $\le t$, then from (\ref{omega}) we have: 
$$
\omega = \prod_{ij} \det(g_{ij}^{d_{ij}}) = (-1)^s\prod_{i<j} \det (g_{ij})^{d_{ij}+d_{ji}}
$$
because $\det g_{ij} = -\det g_{ji}$. So the dependence $\omega$ can be recovered as a function from the values of $d_{ij}+d_{ji}$ and $(-1)^s$. By definition, $|d_{ij}+d_{ji}| \le 2t$ so there are at most $2 \cdot (4t)^{10} < 2^{30} t^{10}$ choices of $\omega$ which are different as functions. 

For each nontrivial $\omega$ of length at most $t$ we bound the number of sequences $A \in \A_{ord}(\alpha, \beta)$ satisfying $\omega(A) = 1$. Suppose that $\omega(A) = 1$ for all $A \in\A_{ord}(\alpha, \beta)$. Consider a sequence $A = (A_1, \ldots, A_5) \in \A_{ord}$, then there exists $h \in \fG$ such that $hA_1 = \alpha, hA_2 = \beta$ so $hA \in \A_{ord}(\alpha, \beta)$. By Claim \ref{dep} and our assumption, $\omega(A) = \omega(hA) = 1$, so $\omega$ is trivial. A contradiction.

Now we note that $\omega$ determines a rational function $\tilde\omega: \FF_p^4 \rightarrow \FF_p$: let $\tilde \omega(x, y) = \omega(\alpha, \beta, x, y, -\alpha-\beta-x-y)$. Each determinant is a degree $2$ polynomial, therefore, $\tilde \omega(x, y) = \frac{P(x, y)}{Q(x, y)}$ where $P$ and  $Q$ have degrees at most $4t$. The number of $A\in \A_{ord}(\alpha, \beta)$ for which $\omega(A) =1$ is less than the number of solutions of the equation $R(x, y) = P(x, y)-Q(x, y)=0$. From the previous paragraph we derive that $R$ is a nontrivial polynomial of degree at most $4t$. By Sparse Zeros Lemma (\cite{BF}, p. 86) $R$ has at most $4t p^3$ roots. 

Altogether, we have $|\A_{ord}^{short}(\alpha, \beta)| \le (2^{30}t^{10}) (4t p^3) \le 2^{34} n^{1.61} = O(n^{1.7})$. 
\end{proof}

Now  we prove that short trivial dependencies are indeed ``trivial".

\begin{lemma}\label{trivial}
Let $\omega = (d_{ij})$ be a trivial dependence of length at most $t = n^{0.01}$. Then $d_{ij}+d_{ji} = 0$ for any $i \neq j$ and the sum $D = \sum_{i < j} d_{ij}$ is even. 
\end{lemma}

\begin{proof}
As we have mentioned before, we have
$$
\omega(A) = \prod_{i \neq j} \det(g_{a_i, a_j})^{d_{ij}} = (-1)^D \prod_{i < j} \det(g_{a_i, a_j})^{d_{ij} + d_{ji}}.
$$
Analogously to the previous lemma, $\omega$ may be written as a fraction $\frac{P(x_1, \ldots, x_4)}{Q(x_1, \ldots, x_4)}$ of polynomials in $8$ variables of degrees at most $4t$ (each $x_i$ represents a vector of two variables $(x_i^1, x_i^2)$). As before, we consider the polynomial $R=P-Q$. Since the dependence $\omega$ is trivial, for any vectors $x_1, x_2, x_3, x_4 \in V$ such that $(x_1, x_2, x_3, x_4, -x_1-x_2- x_3-x_4) \in \A_{ord}$ it follows that $R(x_1, x_2, x_3, x_4) = 0$. Condition $(x_1, x_2, x_3, x_4, -x_1-x_2- x_3-x_4) \in \A_{ord}$ means that these five vectors are in general position. An easy calculation yields that the number of such tuples is at least $p^8-10p^7 > 4tp^7$ so by Sparse Zeros Lemma $R$ must vanish, i.e. $1 \equiv \frac{P}{Q} \equiv (-1)^D \prod_{i < j}\det(x_i, x_j)^{d_{ij} + d_{ji}}$. But determinants $\det(x_i, x_j)$ are pairwise coprime for $i < j$ so each multiple must be equal to $1$ (indeed, $\det(x_i, x_j)$ is an irreducible polynomial of degree $2$ in variables $x_i^1, x_i^2, x_j^1, x_j^2$; two different polynomials $\det(x_i, x_j)$ can not be proportional). The lemma follows. 

\end{proof}

\subsection{Construction of $\A^{wall}$}\label{ss4}

In this section we prove the following crucial lemma:

\begin{lemma} \label{wall}
There are sets $\A^{wall}_{ord} \subset \A_{ord}$ and $\A^{wall}  = \pi(\A^{wall}_{ord})$ such that: 

1. Choose any $A \in \A_{ord}^{long}$ and $g_1, \ldots, g_m \in E$. Suppose that the dependence $\omega = \det(g_1 g_2 \ldots g_m)$ is of length at least $t/3$ and $\omega(A)=1$. Then for some $i$ we have $g_1 g_{2} \ldots g_i(A)A \in \A^{wall}_{ord}$.

2. For any pair of vectors $\alpha, \beta$ we have $|\A^{wall}(\alpha, \beta)| = o(n^2)$.
\end{lemma}

%Basically, Lemma \ref{wall} tells us that we can remove from $\A$ a small set of vertices in such a way that in any ``generic" orbit all ``nontrivial" cycles are destroyed. 
The idea behind the proof is very simple: in each orbit, we sample a random set of relatively small ``boxes" which will almost surely cover most of the orbit. By our assumption that the orbit is in $\A^{long}$ we conclude that a ``nontrivial" cycle does not fit in a box, so it must intersect its boundary. We put $\A^{wall}$ to be the union of the boundaries of the sampled boxes.

\begin{proof}
Consider the orbit decomposition of $\A_{ord}^{long}$ under the action of $\fG$: $\A_{ord}^{long} = \A_1 \cup \ldots \cup \A_l$. Note that by Lemma \ref{short} $|\A_{ord}^{long}| \sim |\A_{ord}| \sim n^4$, also we know that $|\fG| \sim n^2$, consequently, $l \sim \frac{n^4}{|\fG|} \sim n^2$. Choose a representative $A_j \in \A_j$ for each $j$.

We start by constructing $\A^{wall}_{ord}$ in each orbit separately:

\begin{claim}\label{boxy}
For any $j=1, \ldots, l$ there is a set $\A_j^{wall} \subset \A_j$ such that $|\A_j^{wall}| = O(n^2 t^{-0.5})$ and the following holds for any $A \in \A_j$ and $g_1, \ldots, g_m \in E$. Suppose that the dependence $\omega = \det(g_1 g_2 \ldots g_m)$ is of length at least $t/3$ and $\omega(A)=1$. Then for some $i$ we have $g_1 g_{2} \ldots g_i(A)A \in \A^{wall}_j$.
\end{claim}

\begin{proof}
Take a representative $A \in \A_j$, consider two sets $T = \{ \det g_{ij}(A)~|~i\neq j\}$ and $\tilde T = \{ \det g_{ij}(A)~|~i < j\}$ both lying in $\FF_p$. Note that $|T| = 20$ and $|\tilde T| = 10$ because $A \in \A_{ord}^{long}$. 
For a positive integer $\lambda$ let us define a box $B_\lambda$ and the boundary of the box $\partial B_\lambda$ as follows:

\begin{eqnarray*}
B_\lambda = \left \{ \pm \prod_{a \in \tilde T} a^{\lambda_a}~|~ \lambda_a\in [-\lambda, \lambda], ~\sum_{a \in \tilde T} \lambda_a = 0 \right \}, \\
\partial B_{\lambda} = \left \{ \pm \prod_{a \in \tilde T} a^{\lambda_a}~|~ \lambda_a\in [-\lambda, \lambda],~\sum_{a \in \tilde T} \lambda_a = 0, ~ \exists b:~\lambda_b = \pm \lambda \right \}.
\end{eqnarray*}

Here are some properties of these objects:

\begin{claim}\label{box} If $\lambda \le t/30$ then:

1. $|B_\lambda| \ge (2\lambda)^{10}$ and $|\partial B_\lambda| \le 2^{16} \cdot \lambda^9$.

2. All products of the form $\pm \prod_{a \in \tilde T} a^{\lambda_a}$ where $\lambda_a\in [-\lambda, \lambda], ~\sum_{a \in \tilde T} \lambda_a = 0$ are distinct. 
\end{claim}

\begin{proof} Note that the bound on $|\partial B_\lambda|$ is obvious and the bound on $|B_\lambda|$ is an immediate consequence of the Part 2. 

Suppose that two different products of the given form do coincide. Bringing everything on the left hand side we obtain an equality $\pm \prod_{a \in \tilde T} a^{\lambda_a} = 1$  where $\lambda_a \in [-2\lambda, 2\lambda]$ and $\sum \lambda_a = 0$. Expressing this in terms of $g_{ij}$ we obtain an equation $\pm \prod_{i > j} \det(g_{ij})^{\lambda_{ij}}(A)=1$ where $\lambda_{ij}$ obey the same conditions. Construct a dependence $\omega$ in the following way: let $\omega = (\lambda_{ij})$ if the sign before the product is positive and let $\omega = (\lambda_{ij}) \cdot \det(g_{12}g_{21}^{-1})$ if the sign is negative.

%but we should take care of $\pm$ multiple invoking the equation $\det g_{ij} = -\det g_{ji}$) 
We see that $\omega(A)=1$ and $\omega$ has length at most $10 \cdot 2\lambda + 2 < t$. Since $A \in \A_{ord}^{long}$ we deduce that $\omega$ is trivial, and so each $\lambda_{ij} = 0$, a contradiction because initially we chose two different products (note that we considered only $\lambda_{ij}$ with $i < j$).

\end{proof}
Now we fix $\lambda = t / 1000$ and $q = \frac{p}{\lambda^{9.5}}$. Choose independently at random $q$ nonzero residues $\rho_1, \ldots, \rho_q \in \FF_p^{\times} = \FF_p \setminus 0$ and consider random sets
$$
C = \bigcup_{i = i}^{q} \rho_i \cdot B_\lambda, ~~ \partial C = \bigcup_{i = i}^{q} \rho_i \cdot \partial B_\lambda, ~~ R = \FF_p^\times \setminus C.
$$
Let us bound their cardinalities. Every residue does not belong to $\rho_i B_\lambda$ with probability $1- \frac{|B_\lambda|}{p-1}$ so
$$
\mathbb E|R| \le p \left ( 1 - \frac{|B_\lambda|}{p-1} \right)^q  < p  \left ( 1 - \frac{(2\lambda)^{10}}{p-1} \right)^q < p e^{- \frac{q\lambda^{10}}{p}} = O(p e^{-\sqrt{\lambda}}).
$$
Next, $|\partial C| \le q|\partial B_\lambda| = O(q \lambda^9) = O(\frac{p}{\sqrt{\lambda}})$. Therefore, there is a choice of $\rho_i$-s so that $|\partial C \cup R| = O(p \lambda^{-0.5})$. Let $\A_j^{wall}$ be the set of all $hA \in \A_j$ such that $\det h \in \partial C \cup R$. We claim that this is the right choice of $\A_j^{wall}$.

First, it is straightforward that $|\A_j^{wall}| = O(n^2 t^{-0.5})$. Next, take an arbitrary $hA \in \A_j$ and elements $g_1, \ldots, g_m \in E$ such that $\omega = \det(g_1 \ldots g_m)$ is of length at least $t/3$ and $1=\omega(hA)=\omega(A)$. Define dependencies $\omega_i = \det(g_1 \ldots g_i)$ and note that lengths of $\omega_i$ and $\omega_{i+1}$ differ by at most $2$. So there is $k \in [1, m]$ such that the length of $\omega_k$ lies in the interval $(10\lambda, t/60]$. Thus, by Claim \ref{box} $\omega_k(A) \in B_{t/30}$ and it has a unique product representation, so $\omega_k(A) \not \in B_{2\lambda}$ because otherwise it would have length at most $10 \times 2\lambda / 2$ (the length is the half of the sum of degrees). We conclude that residues $\det h$ and $\omega_k(A)\det h$ can not lie in the same homothetic image of $B_\lambda$. So there exists a maximal number $i$ such that $\omega_m(A)\det h = \det h$ and $\omega_i(A)\det h$ lie in different boxes (unless they do not lie in any box $\rho B_\lambda$ at all, in which case we are done). Suppose that $\det h, \omega_{i+1}(A)\det h \in \rho_f B_\lambda$ but $\omega_{i}(A)\det h \not \in \rho_f B_\lambda$. It follows that $\omega_{i+1}(A)\det h \in \rho_f \partial B_\lambda$ and so $g_{1}\ldots g_{i+1}(hA)hA \in \A_j^{wall}$ but this is what we needed to prove. The claim is proved.

\end{proof}

Now we glue $\A_{ord}^{wall}$ from $\A_j^{wall}$. In view of Claim \ref{boxy}, we only need to ensure that all sets $\A_{ord}^{wall}(\alpha, \beta)$ are small. Let us choose uniformly at random operators $\gamma_j \in \fG$ for $j = 1, \ldots, l$ and consider sets $\mathcal B_j = \gamma_j \A_j^{wall}$ and $\mathcal B = \bigcup \mathcal B_j$. We claim that with high probability we can take $\A_{ord}^{wall} := \mathcal B$, so we only need to prove that $|(\pi \mathcal B)(\alpha, \beta)|$ is small for any pair of linearly independent vectors $\alpha, \beta$. Define $\xi_{\alpha, \beta, j} = |(\pi \mathcal B_j)(\alpha, \beta)|$ and $\xi_{\alpha, \beta} = |(\pi \mathcal B)(\alpha, \beta)|$. Clearly, $\xi_{\alpha, \beta} \le \sum_{j} \xi_{\alpha, \beta, j}$ and for any $j$ $\xi_{\alpha, \beta, j} \le 20$ because there is no two elements $h_1 A, h_2 A \in \A_j$ which contain $\alpha, \beta$ on the same places. Let $Q = \frac{n^2}{\sqrt{\log n}}$. As variables $\xi_{\alpha, \beta, j}$ are independent and probability that $\xi_{\alpha, \beta, j}>0$ is at most $P=\frac{|\mathcal A_j|}{p} = O(\frac{1}{\sqrt{t}})$ we conclude that

$$
\mathbb P(\xi_{\alpha, \beta} > Q) < {l \choose Q/20}P^{Q/20} < 2^l O(t^{-0.5})^{Q/20} < e^{c_1 n^2 - c_2 Q\log n} = O(e^{-cn^2 \sqrt{\log n}}) = O(n^{-10}),
$$
(the condition $\xi_{\alpha, \beta} > Q$ implies that there are at least $Q/20$ nonzero $\xi$-s). Thus, with high probability $\xi_{\alpha, \beta} \le Q$ for all $\alpha, \beta$. The lemma is proved.

\end{proof}
\subsection{Proof of Lemma \ref{lm1}}\label{ss5}

Let us consider the decomposition of $\A^{long}$ into the orbits $\A_1, \ldots, \A_l$ under the action of $\fG$ (note that this is not the same decomposition as in the proof of Lemma \ref{wall} because now we work inside $\A$ instead of $\A_{ord}$). We have $|\A^{long}| \sim \frac{n^4}{120}$ and $\fG \sim n^2$ so $l \sim \frac{n^2}{120}$. Thus, to prove Lemma \ref{lm1} we only need to color each set $\A_j \setminus \A^{wall}$ in $20$ colors.  

\begin{lemma} \label{crucial}
For any $j = 1, \ldots, l$ we have $\chi(G|_{\A_j \setminus \A^{wall}}) \le 20$.
\end{lemma}  

\begin{proof}
Let us fix $\A_j$ and its representative $A \in \A_j$ and $A_{ord} \in \pi^{-1}A$. Two vertices $gA$ and $hA$ are adjacent in $G$ if and only if $|gA \cap hA| \ge 2$ or equivalently $|g^{-1}hA \cap A| \ge 2$, that is, some two elements of $A$ are mapped by $g^{-1}h$ into other two elements of $A$. By definition, this means that $g^{-1}h = e(A_{ord})$ for some $e \in E$ (see Section \ref{ss2}), that is $h = g e(A_{ord})$. The induced subgraph $G|_{\A_j \setminus \A^{wall}}$ splits into connected components $\mathcal C_1, \ldots, \mathcal C_m$. Take a representative $h_i A$ from the component $\mathcal C_i$. 

Every vertex $hA \in \mathcal C_i$ has a representation 
\begin{equation}\label{rep}
hA = h_i g_1 f_1^{-1} g_2 f_2^{-1} \ldots g_q f_q^{-1} A,
\end{equation}
where $g_s, f_s \in S(A) := \{ g_{ij}(A_{ord}) \}$ and for all $s$ the vertex $h_i g_1 f_1^{-1} \ldots g_s f_s^{-1} A$ lies in $\mathcal C_i$. Consider an arbitrary bijection $\psi: S(A) \rightarrow \mathbb Z_{20}$ such that $\psi(g_{ij}(A)) =  \psi(g_{ji}(A))+10 \pmod{20}$. Now we define a coloring $c$ of $\A_j \setminus \A^{wall}$ as follows: 
$$
c(h) = \sum_{s=1}^{q} \psi(g_s) - \psi(f_s) \pmod{20}.
$$
Let us suppose for a moment that this definition does not depend on the choice of the representation (\ref{rep}) of $hA$. Then we can write any two adjacent vertices $hA, h'A$ in the form:
\begin{eqnarray*}
hA = h_i g_1 f_1^{-1} g_2 f_2^{-1} \ldots g_q f_q^{-1} A \\
h'A = h_i g_1 f_1^{-1} g_2 f_2^{-1} \ldots g_q f_q^{-1} g'(f')^{-1} A,
\end{eqnarray*}
and so $c(h') - c(h) = \psi(g') - \psi(f') \neq 0 \pmod{20}$ that is the coloring $c$ is proper. 

So it remains to check correctness of the definition of the coloring $c$. Take a vertex $h A \in \mathcal C_i$ and two its representations of the from (\ref{rep}). Bringing everything to the left hand side we obtain
$$
h_i x_1y_1^{-1} x_2y_2^{-1} \ldots x_u y_u^{-1} A = h_iA,
$$
for some $x_s, y_s \in S(A)$ and we need to prove that $\sum_{s=1}^{u} \psi(x_s) - \psi(y_s) = 0 \pmod{20}$. Note that for any $s$ the vertex $h_i x_1 y_1^{-1} \ldots x_s y_s^{-1} A$ has to lie inside $\mathcal C_i$. We can write each multiple $x_sy_s^{-1}$ as $e_s (A_{ord})$ for some $e_s \in E$. Consider the dependence $\omega = \det(e_1 \ldots e_u) =: (d_{ij})$. By definition, $\omega(A_{ord}) = \prod \det(x_sy_s^{-1}) = 1$ and $\sum_{s=1}^{u} \psi(x_s) - \psi(y_s) = \sum_{i \neq j} d_{ij} \psi(g_{ij}) \pmod{20}$. 

Suppose that the length of $\omega$ is at most $t$. Then $\omega$ is trivial because $A_{ord} \in \A^{long}_{ord}$. Then Lemma \ref{trivial} applies and we obtain:
$$
\sum_{i \neq j} d_{ij} \psi(g_{ij}) = \sum_{i < j} d_{ij}\psi(g_{ij}) + d_{ji} \psi(g_{ji}) = \sum_{i < j}\psi(g_{ji}) (d_{ji} + d_{ij}) + 10 \cdot d_{ij} = 10 \sum_{i < j} d_{ij} = 0 \pmod{20}
$$
and we are done. 

Now suppose that the length of $\omega$ is at least $t > t/3$. Then Lemma \ref{wall} applied to the sequence $e_1, \ldots, e_u$ and $h_i A_{ord}$ yields that there exists $s$ such that $h_i e_1 \ldots e_s(A_{ord}) A_{ord} \in \A^{wall}_{ord}$. But this means that $h_i x_1 y_1^{-1} \ldots x_s y_s^{-1}  A \not \in \mathcal C_i$ because $\mathcal C_i \cap \A^{wall} = \emptyset$. We arrived at a contradiction, Lemma \ref{crucial} is proved.
\end{proof}

\subsection{Proof of Lemma \ref{lm2}}\label{ss6}

We should color all $4$-element subsets which are not subsets of any element of $\A_{good}$. Take an arbitrary $X=\{x_1, x_2, x_3, x_4\} \subset V\setminus 0$ and denote $A = X \cup \{ -x_1 -x_2-x_3-x_4\}$. Let $U_1$ be the set of all $4$-element sets $X$ such that $A \not \in \A$, that is a pair of elements of $A$ is collinear (this includes the cases then the sum of $x_i$-s equals $0$). Finally, let $U_2$ be the set of all $4$-element sets $X$ such that $A \in \A^{short}\cup \A^{wall}$. Clearly, $U=U_1 \cup U_2$ and we need to show that $\chi(G(n, 4, 2)|_{U_i}) = o(n^2)$ for $i =1, 2$. Note that in this section we are working with $4$-element subsets of $V \setminus 0$ and, in particular, $U_1, U_2 \subset {V \setminus 0 \choose 4}$. 

Let us begin with $U_2$, we will deduce the desired bound using the inequality $\chi(G(n, 4, 2)|_{U_2}) \le \Delta(G(n, 4, 2)|_{U_2})+1$. Take a vertex $X\in U_2$, by Lemmas \ref{short} and \ref{wall} there are $o(n^2)$ sets from $\A^{wall} \cup \A^{short}$ intersecting $X$ in at least two places. Thus, the maximal degree of considered induced subgraph is $o(n^2)$. 

Now we focus on $U_1$. 

Recall that the list chromatic number $\chi_{list}(H)$ of a graph $H$ is the smallest number $k$ such that the following holds. For each assignment of sets $L(v), v \in V(H)$ of cardinality at least $k$ there is a proper coloring $c$ of $H$ such that $c(v) \in L(v)$ for any $v \in V(H)$. We need the following general result.

\begin{claim}\label{general}
Let $H$ be a graph. Suppose that the vertex set of $H$ is covered by a system of subsets: $V(H) = \bigcup_{i=1}^{m} A_i$. Suppose that for any $v \in V(H)$ the number of edges between $v$ and the set $\bigcup_{A_i \not \ni v} A_i$ is at most $d$. Suppose that for any $i$ $\chi_{list}(H|_{A_i}) \le l$ holds. Then $\chi(G) \le l+d$.  

\end{claim}
\begin{proof}
Suppose that we have already colored $B=A_1 \cup \ldots \cup A_i$ in at most $l+d$ colors. Let us show that the set $C = A_{i+1} \setminus B$ can also be colored in at most $l+d$ colors. 
For a vertex $v \in C$ let $L(v)$ be the list of colors in which $v$ can be colored without contradicting the coloring of $B$. By assumption, $v$ has at most $d$ neighbours in $B$, therefore, $|L(v)| \ge l$. 
Since $C \subset A_{i+1}$, we can pick a color from each $L(v)$ so that the resulting coloring of $C \cup B = A_1 \cup \ldots \cup A_{i+1}$ is proper. The claim now follows by induction.
%We have made an induction step, so the claim is proved.
\end{proof}

To apply Claim \ref{general} we construct a covering system of $U_1$ as follows. Let $U^*$ be the set of quadruples $\{x_1, \ldots, x_4\}$ for which $-x_1-x_2-x_3-x_4$ either equals $0$ or is proportional to $x_i$ for some $i$. For a pair of collinear nonzero vectors $\alpha, \beta$ let $U(\alpha, \beta)$ be the set of all $X \in U_1$ containing $\alpha$ and $\beta$. It follows that $U_1 = U^* \cup \bigcup_{\alpha \sim \beta} U(\alpha, \beta)$ (here $\alpha \sim \beta$ means that $\alpha$ and $\beta$ are collinear). Indeed, if $X = \{x_1, \ldots, x_4\} \in U_1$ then either $x_i \sim x_j$ for some $i \neq j$, which means that $X \in U(x_i, x_j)$, either $x_i \sim -x_1-x_2-x_3-x_4$ for some $i$, which means that $X \in U^*$.
Let $d_X$ be the number of edges between $X$ and the subset $\bigcup P$ of $U_1$, where the union is taken over $P \in \{ U^*,  U_{\alpha, \beta}\}$ such that $X \not \in P$.

%$\mathcal U(X) = \bigcup_{P \not \ni X} P$ where $P$ runs through the set $\{ U^*, U_{\alpha, \beta}\}$. 

\begin{claim}
$d_X = O(n^{3/2})$ for any $X \in U_1$.
\end{claim}

\begin{proof}
We count the number of neighbors $Y$ of $X$ in $\bigcup P$. The intersection $Z=X\cap Y$ can be fixed in $6$ ways. There are at most $4np$ sets $Y \in U^*$ containing $Z$. If $Y \in U(\alpha, \beta)$ for $\{\alpha, \beta\} \not \subset X$ then we can choose two last elements of $Y$ in at most $np$ ways. So, altogether, there are at most $O(n^{3/2})$ neighbors of $X$.

\end{proof}

Using a similar argument one can prove that $\Delta(G(n, 4, 2)|_{U^*}) = O(n)$ and, by the trivial inequality $\chi_{list}(H) \le \Delta(H)+1$, we obtain $\chi_{list}(G(n, 4, 2)|_{U^*}) = O(n)$. Now we bound the list chromatic number of the graph induced on $U(\alpha, \beta)$. Clearly, this subgraph is isomorphic to a subgraph of $G(n, 2, 0)$. So it remains to bound the list chromatic number of the latter graph. We will prove a slightly more general result.
\begin{lemma} \label{list} 
Fix $r, s$ and  let $n \rightarrow \infty$, then 
$\chi_{list}(G(n, r, s)) = O(n^{s+1} \log n) $.
\end{lemma}

\begin{proof}
The assertion follows immediately from the standard fact that $\chi_{list}(G) \le \chi(G) \log |G|$ for any graph $G$.

%Fix $m = rn^{s+1} \log{n}$. Consider arbitrary lists $L(A), A \in {[n] \choose r}$ of cardinality $m$ which are embedded in some set $\Omega$.

%Consider a uniformly random map $\varphi: \Omega \rightarrow {[n] \choose s+1}$. Let $B_A$ be an event that $\varphi(L(A)) \cap {A \choose s+1} = \emptyset$. We have
%$$
%\mathbb{P}(B_A) = \left ( 1 - \frac{{r \choose s+1}}{{n \choose s+1}}\right )^{m} \le \exp \left ( -\frac{m}{n^{s+1}} \right) < n^{-r}.
%$$
%Therefore, the probability of the union of events $B_A$ is at most ${n \choose r} n^{-r} < 1$ so there is $\varphi$ such that all events $B_A$ do not hold. Color a set $A \in {[n] \choose r}$ in any color $c \in L(A)$ for which $\varphi(c) \in {A \choose s+1}$. All sets of the color $c$ contain the set $\varphi(c)$, i.e. intersect in at least $(s+1)$ places, thus, they form an independent set. 

\end{proof}

Thus, we checked assumptions of Claim \ref{general} with $d, l = O(n^{3/2})$. Therefore, $\chi(G(n, 4, 2)|_{U_1}) = o(n^2)$.

\section{Remarks} \label{remarks}

In this Section we very briefly discuss the limitations of the presented techniques. 

Unfortunately, the presented approach depends heavily on the particular choice of the parameters $(r, s) = (4, 2)$ and it is not completely clear how one can extend it to larger values of $r, s$ or to other families of graphs. 

In particular, if one tries to generalize the method to an arbitrary graph $G(n, r, s)$ with $r \le 2s+1$ one faces the following problems:

{\bf 1.} What should replace the set $\A$, that is an approximate design which is invariant under an action of a large group (more precisely, a group of large transitivity, for instance, $GL_t(\FF_p)$)? It is only straightforward to come up with such a family in the case then $s = r-2$: identify $[n]$ with $\FF_p^s \setminus 0$ and consider a family of sets which elements sum up to $0$. The action of the general linear group $GL_s(\FF_p)$ will preserve this family.

{\bf 2.} During the proof, it was essential to work with determinants of linear operators instead of operators themselves. For instance, it allowed us to use polynomial methods in Section \ref{ss3}, 
the commutativity of $\FF_p^\times$ allowed us to construct the ``wall" in Section \ref{ss4}: indeed, the key observation was that the size of the boundary of a box is much smaller than the size of the box itself. It is completely false if one tries to apply this idea to $GL_2(\FF_p)$ directly. But for $s \ge 3$ the determinant is not strong enough to capture all the edges of the graph: the set $E$ of linear operators in $\FF_p^s$ which map a fixed set $A$ to a set which intersects $A$ in at least $s$ elements always contains operators of determinant 1, unless $s = 2$. 

%some adjacent elements of the same orbit will inevitably have the same determinants of the corresponding operators. So any determinant-based coloring will not be able to distinguish them. 

{\bf 3.} Splitting the graph into ``structured" (4-element subsets of elements of $\A^{good}$ in our case) and  ``degenerate" (the set $U$ in our case) parts will also become harder. Consider, for instance, the graph $G = G(n, r, r-2)$. We would like to prove that $\chi(G(n, r, r-2)) \sim \frac{n^2}{6}$. But note that $G(n, r, r-2)$ contains a lot of copies of the graph $G(n-r+3, 3, 1)$ which has almost the same chromatic number and has only $O(n^3)$ vertices. This means that the ``degenerate" part of the graph should be very carefully defined so that it will not accidentally contain a copy of $G(n-r+3, 3, 1)$ inside. Note that the case $(r, s) = (4, 2)$ the set $U$ contains many subgraphs isomorphic to $G(n, 2, 0)$ but, luckily for us, the chromatic number of the latter graph grows linearly. 

The most natural candidates for a future generalization of the approach are the graphs $G(n, 5, 2)$ and $G(n, 5, 3)$ both of which represent some of the new difficulties mentioned above. Also one may consider a simpler sequence of graphs, namely graphs $G(n, r, \ge s)$ which have the same sets of vertices as $G(n, r, s)$ but two vertices of $G(n, r, \ge s)$ are connected if their intersection contains {\it at least} $s$ elements. This simplification eliminates the need of approximate designs in the coloring.


\begin{thebibliography} {20}

\bibitem[BF]{BF} Babai L., Frankl P. {\it Linear algebra methods in combinatorics}, Part 1. Department of Computer Science, The University of Chicago, 1992.

\bibitem[BKR]{BKR} J. Balogh, A.V. Kostochka, A.M. Raigorodskii, {\it  Coloring some finite sets in $ {\mathbb R}^n$}, Discussiones Mathematicae
Graph Theory, 33 (2013), N1, 25 - 31.

\bibitem[BKK]{BKK} Bobu AV, Kostina OA, Kupriyanov AE., {\it  Independence numbers and chromatic numbers of some distance graphs}. Problems of Information Transmission. 2015 Apr 1;51(2):165-76.

\bibitem[C]{C} Cherkashin Danila. {\it Coloring cross-intersecting families} // Electronic Journal of Combinatorics. 2018. Vol. 25. No. P1.47. P. 1-9.

\bibitem[CKR]{CKR} D. Cherkashin, A. Kulikov, A. Raigorodskii, {\it On the chromatic numbers of small-dimensional Euclidean spaces}, 
Discrete and Applied Math., 243 (2018), 125 - 131. 

\bibitem[FF]{FF} Frankl, Peter, and Zoltán Füredi. {\it Forbidding just one intersection.} Journal of Combinatorial Theory, Series A 39.2 (1985): 160-176.

\bibitem[FW]{FW} P. Frankl, R. Wilson, {\it  Intersection theorems with geometric consequences}, Combinatorica, 1 (1981), 357 - 368.

\bibitem[KK]{KK} Kahn J, Kalai G., {\it A counterexample to Borsuk’s conjecture.} Bulletin of the American Mathematical Society. 1993;29(1):60-2.

\bibitem[K1]{K1} Keevash P., {\it  The existence of designs}. arXiv preprint arXiv:1401.3665. 2014 Jan 15.

\bibitem[K2]{K2} Keevash P., {\it  The existence of designs II}. arXiv preprint arXiv:1802.05900. 2018 Feb 16.

\bibitem[KL]{KL} Keller, Nathan, and Noam Lifshitz. {\it The junta method for hypergraphs and Chvátal’s simplex conjecture.} arXiv preprint arXiv:1707.02643 (2017).

\bibitem[KBC]{KBC} Kiselev S., Balogh J., Cherkashin D. {\it Coloring general Kneser graphs and hypergraphs via high-discrepancy hypergraphs} // European Journal of Combinatorics. 2019. Vol. 79C. P. 228-236

\bibitem[KKup]{KKup} Kiselev, S., Kupavskii, A. (2019). {\it Sharp bounds for the chromatic number of random Kneser graphs}. Acta Mathematica Universitatis Comenianae, 88(3), 861-865

\bibitem[Kup]{Kup} A. Kupavskii, {\it Random Kneser graphs and hypergraphs}, Electronic Journal of Combinatorics (2018) P4.52

\bibitem[L]{L} Lov{\' a}sz, L{\' a}szl{\' o}. {\it Kneser's conjecture, chromatic number, and homotopy.} Journal of Combinatorial Theory, Series A 25.3 (1978): 319-324.

\bibitem[N]{Nagy} Z. Nagy, {\it  A certain constructive estimate of the Ramsey number}, Matematikai Lapok, 23 (1972), N 301-302, 26.

\bibitem[P]{P} M. Pyaderkin,  {\it  On the chromatic number of random subgraphs of a certain distance graph}, Discrete and Applied Math (2018).

\bibitem[R]{R} A.M. Raigorodskii, {\it Combinatorial geometry and coding theory}, Fundamenta Informatica, 145 (2016), 359 - 369. 

\bibitem[R{\" o}]{Ro} R{\" o}dl, Vojt{\v e}ch. "On a packing and covering problem." European Journal of Combinatorics 6.1 (1985): 69-78.

\bibitem[S]{S} Sidorenko, Alexander. {\it What we know and what we do not know about Turán numbers.} Graphs and Combinatorics 11.2 (1995): 179-199.

\bibitem[SR]{SR} Sagdeev, A., Raigorodskii, A. (2019). {\it On a Frankl-Wilson theorem and its geometric corollaries}. Acta Mathematica Universitatis Comenianae, 88(3), 1029-1033.

\bibitem[WS]{WS} F.J. MacWilliams, N.J.A. Sloane, {\it The theory of error-correcting codes}, North-Holland, Amsterdam, 1977.

\bibitem[Z]{Z} Zakharov D., {\it On chromatic numbers of some distance graphs.} arXiv preprint arXiv:1608.01873. 2016 Aug 5, to appear in Math. Notes

\bibitem[ZR]{ZR} D. A. Zakharov, A. M. Raigorodskii, {\it Clique Chromatic Numbers of Intersection Graphs}, Math. Notes, 105:1 (2019)
\end{thebibliography}
\end{document}